\documentclass[a4paper]{article}
\usepackage[english]{babel}

\usepackage{amsmath,amssymb,amsthm,exscale,amsopn,mathtools,fullpage}
\usepackage{pstricks, pst-node, pst-text, pst-3d}
\usepackage{pstricks-add}
\usepackage{mathdots}
\usepackage{tikz}
\usepackage{tikz-cd}
\usepackage{amssymb}
\usepackage{amsmath}
\usepackage{polski}
\usepackage{color}
\usepackage{graphicx}
\usepackage{rotating}
\usepackage{epstopdf}
\usepackage{pstricks, pst-node, pst-text, pst-3d}
\usepackage{pstricks-add}
\usepackage{mathdots}

\usepackage[left=2.5cm,top=3cm,right=3cm,bottom=2cm]{geometry}

\newtheorem{lemma}{Lemma}[section]

\newtheorem{theorem}[lemma]{Theorem}
\newtheorem{obs}[lemma]{Observation}
\newtheorem{coll}[lemma]{Corollary}

\newtheorem{defi}[lemma]{Definition}

\DeclareMathOperator{\HK}{HK}
\DeclareMathOperator{\GKdim}{GKdim}
\DeclareMathOperator{\supp}{supp}

\DeclareMathOperator{\grecd}{gcd}

\newcommand{\s}{\subseteq}
\numberwithin{equation}{section}

\title{}
\date{}

\author{}

\begin{document}

    \title{The Gelfand--Kirillov dimension of Hecke--Kiselman
    algebras}

    \author{Magdalena Wiertel}
    \date{}
    \maketitle

    \begin{abstract}
    Hecke--Kiselman algebras $A_{\Theta}$, over a field $k$, associated to finite
    oriented graphs $\Theta$ are considered. It has been known that every such algebra
    is an automaton algebra in the sense of Ufranovskii. In particular, its
    Gelfand--Kirillov dimension is an integer if it is finite.
    In this paper, a numerical invariant of the graph $\Theta$ that determines
    the dimension of $A_{\Theta}$ is found.
    Namely, we prove that the Gelfand--Kirillov dimension of $A_{\Theta}$ is the
    sum of the number of cyclic subgraphs of $\Theta$ and the number of oriented paths
    of a special type in the graph, each counted certain specific number of times.

    \end{abstract}

    \vspace{20pt}



%
%
%
%

    \vspace{30pt}

\section{Introduction}

Let $\Theta = (V(\Theta), E(\Theta))$ be an oriented finite simple
graph with $n$ vertices $x_1, \ldots, x_n$. In \cite{maz} the
Hecke--Kiselman monoid $\HK_{\Theta}$ associated with $\Theta$ has
been defined by the following presentation.
\begin{itemize}
    \item[(i)] $\HK_{\Theta}$ is generated by elements $ x_i^2 = x_i$,
    where $1 \leq i \leq n$,
    \item[(ii)] if the vertices $x_i$, $x_j$ are not connected in  $\Theta$,
    then  $x_ix_j = x_jx_i$,
    \item[(iii)] if $x_i$, $x_j$ are connected by an arrow $x_i \to x_j$ in $\Theta$,
    then $x_ix_jx_i = x_jx_ix_j = x_ix_j$.
\end{itemize}

\noindent By $k[\HK_{\Theta}]$ we mean the monoid algebra of
$\HK_{\Theta}$ over a field $k$. Since the ground field $k$ does
not play any role in our considerations we will denote this
algebra by $A_{\Theta}$.

Hecke--Kiselman monoids are natural quotients of $0$--Hecke
monoids. The monoid algebras of the latter monoids are
specializations of famous Iwahori--Hecke algebras, crucial in the
representation theory of Coxeter groups, \cite{denton},
\cite{huang16}. Moreover, $0$--Hecke monoids have been applied in
algebraic combinatorics, for instance in \cite{marberg}.
Investigation of Hecke--Kiselman monoids and their algebras fits
into the study of various generalizations of algebraic structures
arising from Coxeter groups. Several combinatorial and structural
properties of Hecke--Kiselman monoids and their monoid algebras
have already been studied for example in \cite{for, maz, mo1, mo2,
wo}.

The aim of the present paper is to describe the Gelfand--Kirillov
dimension of Hecke--Kiselman algebras associated to oriented
graphs in terms of numerical invariants of the underlying graph.
This dimension describes an asymptotic behaviour of the growth of
algebras and is a useful tool in the study of noncommutative
algebras. For basic information on the Gelfand--Kirillov dimension
we refer to \cite{krause}.

The following result obtained in \cite{mo1} is our starting point.

\begin{theorem}\label{gkzor} Let $\Theta$ be a finite oriented
simple graph. The following conditions are equivalent.
\begin{itemize}
    \item[(1)] $\Theta$ does not contain two different simple cycles
    connected by an oriented path of length $\geq 0$,
    \item[(2)] $A_{\Theta}$ is an algebra satisfying a polynomial identity,
    \item[(3)] $\GKdim(A_{\Theta}) < \infty$,
    \item[(4)] the monoid $\HK_{\Theta}$ does not contain a
    free submonoid of rank $2$.
\end{itemize}
\end{theorem}
 Our reasoning relies on the property discovered in \cite{mo2},
saying that the Hecke--Kiselman algebras associated to oriented
graphs are automaton algebras. This class of algebras was
introduced as a generalization of algebras with finite Gr\"obner
basis, \cite{ufn}. An automaton algebra is an algebra whose set of
normal forms is recognised by a finite deterministic automaton.
This property implies several consequences for the combinatorics
and growth of the algebra. In particular, the Gelfand--Kirillov
dimension of automaton algebras is an integer if it is
finite, \cite{ufn}, and can be expressed using certain forms of
elements of the algebra, \cite{automata}. To obtain our main
result in Theorem~\ref{main} we investigate the combinatorics of
words in Hecke--Kiselman monoids. Gr\"obner bases of
Hecke--Kiselman algebras described in \cite{mo2} and the
characterization of almost all elements of the monoid associated
to an oriented cycle of any length, obtained in \cite{wo}, are
extensively used in our approach.

\section{Preliminaries}\label{prelim}

\noindent Following \cite{ufn}, let us recall some definitions that concern combinatorics on words in the context of finitely generated algebras.\\

\noindent Let $F$ denote the free monoid on the set $X$ of $n\ge
3$ free generators $x_1,\dotsc,x_n$. Let $k$ be a field and let
$k[F] =k\langle x_1,\dotsc,x_n\rangle$ denote the corresponding
free algebra over $k$. For every $x \in X$ and $w \in F$ by
$|w|_x$ we mean the number of occurrences of $x$ in $w$. By $|w|$
we denote the length of the word  $w$. The support of $w \in F$,
denoted by $\supp(w)$, stands for the set of all $x \in X$ such
that $|w|_x > 0$. We say that the word $w = x_{i_1}\cdots x_{i_r}
\in F$ is a subword of the word $v \in F$, where $x_{i_j} \in X$,
if $ v = v_1x_{i_1}\cdots v_r x_{i_r} v_{r+1}$, for some $v_1,
\ldots, v_{r+1} \in F$. If $v_2, \ldots, v_r$ are trivial words,
then we say that $w$ is a factor of $v$. By a prefix (suffix) of
the word $w$ we mean any factor $u\neq 1$ such that $w=uv$ ($w=vu$) for some $v$.\\

\noindent Assume that a well order $<$ is fixed on
$X$ and consider the induced degree-lexicographical order on $F$
(also denoted by $<$). Let $A$ be a finitely
generated algebra over $k$ with a set of generators $r_1, \ldots,
r_n$ and let $\pi: k[F] \to A$ be the natural homomorphism of
$k$-algebras with $\pi(x_i) = r_i$. We will assume that $\ker
(\pi)$ is spanned by elements of the form $w-v$, where $w,v\in F$
 (in other words, $A$ is a semigroup algebra). Let $I$ be the
ideal of $F$ consisting of all leading monomials of $\ker(\pi)$.
The set of normal words corresponding to the chosen
presentation for $A$ and to the chosen order on $F$ is defined by
$N(A) = F \setminus I$. Describing the set $N(A)$ is related to finding a Gr\"obner basis of the ideal $J=\ker(\pi)$ of $k[F]$.
 Recall that a subset $G$ of $J$ is called a Gr\"obner basis of
$J$ (or of $A$) if $0\notin G$, $J$ is generated by $G$ as an ideal and
for every nonzero $f\in J$ there exists $g\in G$ such that the
leading monomial $\overline{g}\in F$ of $g$ is a factor of the
leading monomial $\overline{f}$ of $f$.
If $G$ is a Gr\"obner basis of $A$, then a word $w\in F$ is normal
if and only if $w$ has no factors that are leading
monomials in $g\in G$.\\

Gr\"obner bases of Hecke--Kiselman algebras associated to oriented
graphs have been characterized in \cite{mo2}. For any oriented
graph $\Theta$, $t \in V(\Theta)$ and $w \in F = \langle
V(\Theta)\rangle$ we write $w \nrightarrow t$ if $|w|_t = 0$ and
there are no $x \in \supp(w)$ such that $x \rightarrow t$ in
$\Theta$. Similarly, we define $t \nrightarrow w$: again we assume
that $|w|_t = 0$ and there is no arrow $t \to y$, where $y \in
\supp(w)$. In the case when $t \nrightarrow w$ and $w \nrightarrow
t$, we write $t \nleftrightarrow w$. A vertex $v \in  V(\Theta)$
is called a sink vertex if no arrow begins in $v$. Analogously one
defines a source vertex. Sink and source vertices are called
terminal vertices.

\begin{theorem}[\cite{mo2}]\label{basis} Let $\Theta$ be a finite simple oriented graph with
    vertices $V(\Theta) = \{x_1, x_2, \ldots, x_n\}$. Extend the natural ordering $x_1 < x_2 < \cdots < x_n$
    on the set $V(\Theta) $ to the deg-lex order on the free monoid $F = \langle V(\Theta)  \rangle$.
    Consider the following set $T$ of reductions on the algebra $k[F]$:
    \begin{itemize}
        \item[(i)] $(twt, tw)$, for any $t \in V(\Theta)$ and $w \in F$ such that  $w \nrightarrow t$,
        \item[(ii)] $(twt, wt)$, for any $t \in V(\Theta)$ and $w \in F$ such that  $t \nrightarrow w$,
        \item[(iii)] $(t_1wt_2, t_2t_1w)$, for any $t_1, t_2 \in V(\Theta)$ and $w \in F$
        such that $t_1 > t_2$ and $t_2 \nleftrightarrow t_1w$.
    \end{itemize}
    \noindent Then the set $\{w - v, \text{ where } (w, v) \in T\}$
    forms a Gr\"obner basis of the algebra $A_{\Theta}$.
\end{theorem}

\noindent To emphasize the use of the theorem above, whenever we
consider the set $N(A_{\Theta})$ of normal words of the
Hecke--Kiselman algebra $A_{\Theta} = k[\HK_{\Theta}]$ that is
obtained via reductions from the set $T$, we will say that the
elements of $N(A_{\Theta})$ are the reduced words of $A_{\Theta}$.

This result leads to the following corollary, obtained in \cite{mo2}, that will be useful in calculation of the Gelfand--Kirillov dimension of Hecke--Kiselman algebras.

 \begin{theorem}\label{autom} Assume that $\Theta$ is a finite simple oriented graph. Then
    $A_{\Theta}$  is an automaton algebra, with respect
    to any deg-lex order on the underlying free monoid of rank $n$.
    Consequently, the Gelfand--Kirillov dimension $\GKdim(A_{\Theta})$ of
    $A_{\Theta}$ is an integer if it is finite.
 \end{theorem}

\noindent Recall that $A$ is an automaton algebra if $N(A)$ is a
regular language. That means that this set is obtained from a
finite subset of $F$ by applying a finite sequence of operations
of union, multiplication and operation $*$ defined by $T^* =
\bigcup_{i \geq 0}T^i$, for $T \subseteq F$. Similarly, we define
$T^+ = \bigcup_{i \geq 1}T^i$ for $T \subseteq F$. If $T = \{w\}$
for some $w \in F$, then we write $T^* = w^*$ and $T^+ = w^+$. An
expression built recursively from the set of letters from $F$
using operations of union, multiplication and $*$ is called a
regular expression. The importance of automaton algebras comes
from the deep results from the theory of automata. Recall that a
finite automaton is an oriented graph with two distinguished sets
of vertices (possibly intersecting), called initial and final
states and with edges labelled with letters of a finite alphabet
$X$. An automaton is called a deterministic automaton, if there is
only one initial vertex and, at every vertex, for every letter,
there exists a unique edge beginning with that vertex and marked
by that letter. The language defined by an automaton consists of
the set of all the words formed by reading through a path from any
initial vertex to any final vertex. The famous Kleene's theorem
states that every regular language may be defined by a
deterministic automaton. This property is especially useful in the
case of
automaton algebras of finite Gelfand--Kirillov dimension, as we will see below.\\

\noindent It is known that the GK-dimension of an automaton
algebra is either infinite or an integer, see for example
Theorem~3 on page 97 in~\cite{ufn}. Moreover, in the finite
dimensional case, the dimension is related to certain forms of
regular-expressions representations of the regular languages of
normal words, \cite{automata}. We reformulate the results of
\cite{automata} to apply them in the case of Hecke--Kiselman
algebras.

Let us start with the necessary notations and remarks. The density
function of a regular language $L\subseteq F$, where $F$ is the
free monoid over the set $X$ is defined as $p_{L}(n)=|L\cap
X^{n}|$, that is the number of elements in $L$ of length $n$.
Given two functions $f(n)$ and $g(n)$, we say that $f(n)$ is
$O(g(n))$ if there are positive constants $C$ and $n_0$ such that
$f(n)\leqslant Cg(n)$ for every $n\geqslant n_0$. In particular,
the density function of the regular language $N(A)$ of normal
words of an automaton algebra $A$ is $O(n^{k-1})$ for some
$k\geqslant 1$ precisely when the growth of $A$ is $O(n^{k})$ and,
consequently, $\GKdim(A)\leqslant k$. Therefore the following can
be obtained as a consequence of Theorem~3 in \cite{automata}.

\begin{theorem}[\cite{automata}, Theorem~3]\label{automata-est} The Gelfand--Kirillov
dimension of an automaton algebra $A$ is not bigger than $k$ for some $k\geqslant 0$ if
and only if the set of normal words $N(A)$ can be represented as a finite union of
regular expressions of the following form
\begin{equation}v_0w_{i_1}^*v_1w_{i_2}^*v_2\ldots v_{s-1}w_{i_s}^*v_{s},\label{dlugosc}
\end{equation}
with $v_0, \ldots, v_{s}\in F$, $w_{i_1}, \ldots, w_{i_s}\in F$ and $ 0 \leqslant s\leqslant k$.
\end{theorem}

\noindent We aim to determine the Gelfand--Kirillov dimension of
the Hecke--Kiselman algebra $A_{\Theta}$ in the finite-dimensional
case. Due to Theorem \ref{gkzor}, this means that the graph
$\Theta$ does not contain two cycles connected by an oriented
path. As we will show, in this case any family of words of
form \eqref{dlugosc} can be rewritten in such a way that each
$w_{i_j}$ corresponds to a certain cycle $C_j$ in the graph
$\Theta$ and $v_i$ contains a vertex which is connected by an edge with the cycle $C_{i+1}$ for $i\neq s$. To estimate the dimension, we will find the maximal possible $s$ for the words of such form.

Every vertex of $\Theta$ that belongs to some cycle will be called
a cycle vertex, or a cycle generator of $\HK_{\Theta}$. Any vertex
that is not a cycle vertex will be called a non-cycle vertex
(resp. a non-cycle generator).

We end this section with a general observation concerning the
possible forms of elements $w_{i_j}$ in~\eqref{dlugosc}.

\begin{obs}\label{supp}
    Let $\Theta$ be a finite simple oriented graph such that $A_{\Theta}$ is of finite Gelfand--Kirillov
    dimension.  Let $C_{1}, \ldots, C_{k}$ be the set of disjoint simple cycles
    in $\Theta$, where $C_{l}$ is of the form $$x_{1,l} \to x_{2,l} \to \ldots \to x_{n_l,l} \to x_{1,l},$$
    for some $n_l\geqslant 3$ and $1 \leq l \leq k$. Assume any deg-lex order on $F$ such that
    we have $x < y$ for some $x \in C_{r}$ and $y \in C_{s}$ if and only if either
    $r < s$, or if $r = s$ and $x = x_{p, r}, y = x_{q, r}$, for $p < q$.
    Assume that for some $1 \neq w \in F$, the words $w^m \in F$ are reduced
    with respect to the reduction set $T$ in Theorem \ref{basis} (constructed
    with respect to the chosen deg-lex order) for every $m\geqslant 1$. Then $w$ is
    a factor of the infinite word of the form $ (q_{N,i})^{\infty}$ of full support,
    where $x_1 \to x_2 \to \ldots \to x_N \to x_1$ is one of the cycles $C_k$ with $N = n_k$,
    $q_{N,i} = x_{N}(x_{1}\ldots x_{i})(x_{N-1}\ldots x_{i+1})$ and
    $i\in\{0,\ldots, N-2\}$. Here we assume that $q_{N,0} = x_{N}x_{N-1}\ldots x_{1}$.
\end{obs}
\begin{proof}
Let $w\neq 1$ be such that the word $w^m$ is reduced for every
$m\geqslant1$. Suppose that $y\in\supp(w)$ is a non-cycle vertex
of $\Theta$. First, we will show that then the support of $w$
would also contain either a source or sink vertex. If $y$ is not a
terminal vertex, from conditions (i) and (ii) in Theorem
\ref{basis}, it follows that there exist $u_1,z_1\in V(\Theta)$,
$u_1\neq z_1$, such that $u_1\rightarrow y$, $z_1\leftarrow y$ in
$\Theta$ and $u_1,z_1\in\supp(w)$. Similarly, if $u_1$ is not a
sink vertex, then there exists $u_2\in\supp(w)$ such that
$u_2\rightarrow u_1$. Symmetrically, if $z_1$ is not a source
vertex, then $z_2\leftarrow z_1$ in $\Theta$ for some
$z_2\in\supp(w)$. Moreover $\{u_1, u_2\}\cap\{z_1,
z_2\}=\emptyset$, because $y$ is a non-cycle vertex and
$z_2\notin\{y, z_1\}$, $u_2\notin\{y, u_1\}$. We continue this
procedure until at least one of the chosen vertices is either
terminal or cycle vertex. As the graph is finite, after finitely
many steps we obtain a path $ u_s\rightarrow \cdots\rightarrow
u_1\rightarrow y \rightarrow z_1 \rightarrow \cdots \rightarrow
z_r$ such that $u_1,\ldots, u_s,z_1, \ldots, z_r\in\supp(w)$ and
either $u_s$ is a cycle vertex, or a source vertex and, similarly,
either $z_r$ is a cycle vertex, or a sink vertex. From
Theorem~\ref{gkzor} and the assumption that $A_{\Theta}$ is of
finite Gelfand--Kirillov dimension, the graph $\Theta$ does not
contain two cycles connected by a path and thus it follows that
$u_s$ and $z_r$ cannot be both cycle vertices. Therefore, either
$u_s$ is a source or $z_r$ is a sink, as claimed. However,
according to Theorem \ref{basis} a sink or source vertex may occur
in a reduced word at most once. Since $w^2$ is reduced and
contains at least two occurrences of $u_s$ and $z_r$, they cannot
be terminal vertices, which leads to a contradiction.

 \noindent We have proved that the entire support of $w$ consists
of cycle generators. Call these cycles $C_{1}, \ldots, C_{q}$.
Since the Gelfand--Kirillov dimension of $A_{\Theta}$ is finite,
no vertex can belong to two cycles and if two elements in the
support of $w$ belong to different cycles, they are not connected
in $\Theta$ by an oriented path. From Theorem \ref{basis} and from
the assumed deg-lex order on $F$ it follows that $w = w_1w_2\ldots
w_q$, where $\supp(w_p) \subseteq V(C_{i_p})$ for pairwise
different cycles $C_{i_p}$ for $p=1,\ldots, q$. Yet, as $w^m$ is
reduced, for all $m \geq 1$ it easily follows that $q = 1$, so the
support of $w$ belongs entirely to a single cycle. Say that this
cycle $C$ is of the form $x_1 \to x_2 \to \ldots \to x_N \to x_1$.
Suppose that there exists $x_i$ which is not in the support of
$w$. Take an index $i$ such that $x_{i}\notin\supp(w)$ but
$x_{i-1}\in \supp(w)$, where for $i=1$ we take $i-1=N$. Then $w^2$
contains a factor of the form $x_{i-1}ux_{i-1}$ such that
$x_{i}\notin\supp(u)$. From the description of the Gr\"obner basis
in Theorem \ref{basis} it follows that then $w^2$ is not reduced.
This means that $\supp(w)=\{x_1,\ldots, x_N\}$. From \cite{wo},
Proposition~2.14 it follows that if $w^n$ is reduced for every
$n\geqslant 1$, then for some $m\geqslant 1$ the word $w^m$ is of
the form $aq_{N,i}^k b$, where $i\in\{0,\ldots, N-2\}$,
$k\geqslant 1$ and $a$ and $b$ are members of an explicitly
described finite families of words. Then, from the assumption,
$w^{2m}$ has the reduced form $aq_{N,i}^k b aq_{N,i}^k b$. In
particular, this word has a factor $q_{N,i}$ and therefore, from
Theorem~2.1 in~\cite{wo}, it follows that $ba$ is either of the
form $q_{N,i}$ or the trivial word $1$. Consequently, as $w$ is a
prefix and suffix of $w^m = aq_{N,i}^k b$, it is also a factor of
the infinite word of the form $(q_{N,i})^{\infty}$ for some
$i\in\{0,\ldots, N-2\}$. The assertion holds.
\end{proof}

\section{The main result}

Consider an oriented graph $\Theta$ which does not contain two different cycles connected by an oriented path. Then, by Theorem~\ref{gkzor}, the corresponding Hecke--Kiselman algebra $A_{\Theta}=k[\HK_{\Theta}]$ is of finite Gelfand--Kirillov dimension.

If the graph does not contain any oriented cycle, then the corresponding Hecke--Kiselman monoid is finite, see for example \cite{ara}. In this case the Gelfand--Kirillov dimension of the monoid algebra $A_{\Theta}$ is~$0$.

Thus in the present section we assume that $\Theta$ has at least one cycle. Denote the simple cycles of the graph by $C_{1},\ldots, C_{k}$ for some $k\geqslant 1$ and assume that the cycle $C_j$ is of the length $i_j\geqslant 3$ for $j = 1,\ldots, k$.

Let $\Theta'$ be the full subgraph of $\Theta$ whose set of vertices is built from all cycle vertices and all vertices connected with at least one cycle by an oriented path. We will call such a subgraph $\Theta'$ the maximal cycle--reachable subgraph of $\Theta$.

In particular, for any vertex $x\in V(\Theta')$ that is not contained in any cycle, if there exists a path from $x$ to a cycle (from a cycle to $x$, respectively), then all paths between $x$ and all cycles are from $x$ to the cycles (from the cycles to $x$, respectively), as otherwise $\Theta$ would contain two different cycles connected by an oriented path.

Consider any degree-lexicographic order in the free monoid generated by the vertices of $\Theta$. Recall that by reduced words we mean the words that are in the normal form with respect to the set of reductions from Theorem~\ref{basis}.

We start with the estimation of the number of occurrences of
certain non-cyclic vertices in the reduced words of Hecke--Kiselman
monoids. We agree that for any vertex $x$ there exists
exactly one path of length $0$ with the end (or beginning) in $x$.

\begin{lemma}\label{oszacowanie}
    Let $\Theta$ be a finite simple oriented graph with cycles denoted by $C_{1},\ldots, C_{k}$,
    and let $\Theta'$ be its maximal cycle-reachable subgraph. For
    every vertex $x\in V(\Theta')\setminus (V(C_{1})\cup \ldots \cup V(C_{k}))$ either all oriented paths between $x$ and any cycle lead from $x$ into cycles or all lead from cycles into $x$. Denote
    by $k_{x}$ the number of oriented paths in $\Theta$ of non-negative length with the
    end in the vertex $x$ in the first case, and the number of oriented
    paths of non-negative length with the beginning in $x$ in the latter case. Then, in every
    reduced word in $\HK_{\Theta}$, the element $x$ occurs at
    most $k_x$ times.
\end{lemma}

\begin{proof} Let $x$ be any vertex contained in the maximal
    cycle--reachable subgraph $\Theta'$ of the graph $\Theta$ but not
    contained in the cycles $C_{1},\ldots, C_{k}$. Assume first that there are oriented paths from the cycles into $x$. To prove the statement
    we proceed by induction on the maximal length $l(x)$ of a path
    starting at $x$ in the graph $\Theta$.

    If $l(x)=0$ then $x$ is a sink vertex in the graph $\Theta$ and
    thus there are no edges starting at $x$. Then for any
    $w\in\HK_{\Theta}$ we have $xwx=wx$ (see condition (ii) in
        Theorem~\ref{basis}) and thus $x$ can occur at most once in any
    reduced word.

    Assume now that $l(x)>0$ for some $x\in V(\Theta')\setminus
        (V(C_{1})\cup \ldots\cup V(C_{k}))$ and let $z_1,\ldots, z_m$ be the set of all
    vertices in $\Theta$ such that there is an edge $x\rightarrow z_i$
    for every $i=1,\ldots, m$. Then from the definition of the maximal
    cycle--reachable subgraph it follows that all $z_1, \ldots, z_m$
    are also in $\Theta'$. Moreover, for $i=1,\ldots, m$ we have
    $l(z_i) < l(x)$. By the inductive hypothesis every $z_i$ occurs in
    any reduced word at most $k_{z_i}$ times, where
    $k_{z_i}$ is number of paths starting at $z_i$. We know that
    if a word of the form $xwx$ with $|w|_x=0$ is reduced in $\HK_{\Theta}$ then in
    particular $x\rightarrow y$ for some $y\in\supp(w)$, as
    otherwise $x \nrightarrow w $ and $xwx=wx$ in $\HK_{\Theta}$. It
    follows that at least one of $z_1,\ldots, z_m$ occurs between any
    two generators $x$. As already explained, every $z_i$ occurs in any reduced
    word at most $k_{z_i}$ times. Therefore $x$ can occur at most
    $k_{z_1}+\ldots + k_{z_m} + 1$ times in any reduced
    word. On the other hand, in $\Theta$ there is exactly one path of length $0$ starting at $x$. Every
    other path starting from $x$ uniquely determines a path
        starting from one of $z_1,\ldots, z_m$ and every path $p$
    starting at $z_i$ defines a path starting with $x\rightarrow
        z_i$ and followed by $p$. Thus, in total there are exactly
    $k_{z_1}+\ldots + k_{z_m} + 1$ paths starting from $x$ in
    the graph $\Theta$. The assertion follows.

    The case where there exist paths from $x$ to a cycle can be
        treated by a symmetric argument, using induction on the maximal
    length of a path that ends in $x$.

\end{proof}

Note that for every non-cyclic vertex $x$ in the maximal cycle--reachable
subgraph $\Theta'$ such that all paths between $x$ and the cycles
lead from the cycles into $x$ (from $x$ into the cycles,
respectively) the number $k_x$ of all paths in $\Theta$
starting (ending, respectively) at $x$ is the same as the
number of such paths in $\Theta'$.

Our next step is to use Lemma~\ref{oszacowanie} to show that every regular expression of the form $w_{1}^*v_1w_{2}^*\ldots
v_{s-1}w_{s}^*$ which describes reduced words in the algebra $A_{\Theta}$
can be expressed using at most certain number of stars. To do so we need to introduce certain order in the set of vertices of $\Theta$. For the rest of the present section we will assume that such an order had been chosen.

\begin{defi}[Order on vertices of the graph]\label{order}
    Let $\Theta$ be a graph with the cycles $C_{1}, \ldots C_{k}$ of length $n(j)\geqslant 3$ for $j=1,\ldots, k$.
    Denote by $\Theta'$ the maximal cycle--reachable
    subgraph of $\Theta$. As already explained, for any vertex $x$ of $\Theta'$ that is
    not contained in
    any cycle, all oriented paths between $x$ and any of the cycles starts either
    from $x$ or all go to $x$. For every such a vertex denote by $k_x$ the number of oriented paths of length $\geqslant 0$ in $\Theta$ with either the end or the beginning in $x$, depending on the direction of paths between $x$ and the cycles. In the
    set of these vertices define any order such that if $k_x < k_y$
    holds, then $y < x$.

    Let $C_j$ be of the form $x_{1,j}\rightarrow\cdots \rightarrow x_{n(j), j}\rightarrow x_{1,j}$ for some $n(j)\geqslant 3$ and $j=1,\ldots, k$.
    In the set of all cycle vertices introduce the order such that $x_{i,j} < x_{l, m}$
    if ether $j < m$ or $j=m$ and $i<l$. Moreover, assume that all cycle vertices
    are smaller than any vertex outside the cycles.

    Finally, choose any order in the set of vertices of $\Theta$ that
    are not in $\Theta'$, for example such that all these vertices are
    bigger than the vertices of $\Theta'$.
\end{defi}

Let us note that it is possible to define the order which
satisfies all above conditions provided that the graph $\Theta$
does not contain two different cycles connected by an oriented
path.

In the next lemma we describe the possible form of a family of
reduced words described by $w^* v w^*$, with $\supp(w)\s V(C_n)$
for some $n$.

\begin{lemma}\label{non--cyclic}
If a family of reduced words is described by a regular expression
of the form $u^{*}vw^{*}$ with $\supp(u),\supp(w)\s V(C_n)$ for a
cycle $C_n$, then either $v$ contains a vertex connected by
an edge with $C_n$ or this family of words can be expressed by a
sum of finitely many regular expressions of the form $p r^{*}q$ or
$p$, for some words $p, q$ and~$r$.

\end{lemma}
\begin{proof}
Let $u^{*}vw^{*}$ be the regular expression describing reduced
words with $\supp(u),\supp(w)\s V(C_n)$ for a cycle $C_n$. First
we claim that either $\supp(v)\s V(C_n)$ or $v$ contains a
non-cycle vertex. Indeed, by Definition~\ref{order} of the order
on the vertices of $\Theta$ and the fact that the graph does not
contain two different cycles connected by an oriented path,
generators corresponding to the vertices from different cycles
commute. Consequently, every reduced word $w$ such that
$\supp(w)\s V(C_{1})\cup\cdots\cup V(C_{k})$ has elements from
different cycles grouped in such a way that if $w=w_1\cdots w_j$
with $w_i\in V(C_{n(i)})$ and $w_l\in V(C_{n(l)})$, then
$n(i)\leqslant n(l)$ for all $i<l$. Thus if the family of words
$u^*v w^*$ such that $\supp(u), \supp(w)\s V(C_{n})$ is reduced,
then either $\supp(v)\s V(C_n)$ or $v$ contains a non-cycle
generator.

 Let us now consider the first case, that is $u^*v w^*$
consists of reduced words and $\supp(u)$, $\supp(v)$, $\supp(w)\s
V(C_{n})$. We proceed to show that then $u^* v w^*$ can be
expressed as a finite sum of expressions with at most one Kleene
star. From the reasoning as in the proof of Observation~\ref{supp}
it follows that $u, v$ and $w$ are all factors of the infinite
word $(x_{N}(x_{1}\ldots x_{l})(x_{N-1}\ldots x_{l+1}))^{\infty}$,
denoted shortly by $q_{N,l}^{\infty}$, for some $l\in\{0,\ldots,
N-2\}$, where $N$ depends on $n$. Moreover we can write $u = a
q_{N,l}^{\alpha_1}b$, $v = a q_{N,l}^{\beta}b'$ and $u = a'
q_{N,l}^{\alpha_2}b'$ for non-negative $\alpha_i$, $\beta$ and
words $a, a'$, $b, b'$ that are suffixes and prefixes of the word
$q_{N,l}$, respectively, of length at most $N-1$. Thus both $ba$
and $b'a'$ are either the trivial word $1$ or are of the form
$q_{N,l}$. Then $u^*vw^*$ is equal to the set $\{ a q_{N,l}^{l_1
\beta_1 + l_2 \beta_2 + \beta_3} b' : l_1, l_2\geqslant 0\}$, for
some positive integers $\beta_i$ ($i=1,2,3$), where $\beta_1 =
\alpha_1$ if $ba=1$ and $\beta_1 = \alpha_1 + 1$ otherwise, and
$\beta_2 = \alpha_2$ if $b'a'=1$ and $\beta_2= \alpha_2 + 1$
otherwise. From Proposition~2.2 in \cite{num_semigroup} it follows
that there exist a positive integer $n_0$ and a finite set $F$
such that $\{ l_1 \beta_1 + l_2 \beta_2 + \beta_3\} = \{ n_0 + k d
: k\geqslant 0 \}\cup F $, where $d= \grecd{(\beta_1, \beta_2)}$.
We thus get easily that $u^* v w^*$ can be written as a finite sum
of regular expressions with at most one star~$*$.

Now assume that a family of reduced words described by $u^* v w^*$
is such that $\supp(u), \supp(w)\s V(C_n)$ and $v$ contains a
non-cycle vertex. We can write $v = v_{s} z v_{c}$ for words
$v_{s}, v_{c}$ and a non-cycle vertex $z$ such that
$\supp(v_{c})\s \bigcup\limits_{j=1}^{k} V(C_j)$. Suppose that $z$
is not connected by an edge with a cycle $C_n$. Consider the first
occurrence of a vertex $x$ such that $x\in V(C_n)$ in the
word $v_cw$. Then the word $vw$ contains a factor of the
form $zv'x$ with $\supp(v_{c})\s \bigcup\limits_{j\neq n}V(C_j)$.
Furthermore, $x < z$ and $zv'\nleftrightarrow x$. Consequently,
$vw$ contains a factor which can be reduced using reduction (iii)
from Theorem~\ref{basis}. The obtained contradiction shows that
for every family of reduced words of the form $u^* v w^*$ with
$\supp(u), \supp(w)\s V(C_n)$ and $\supp(v)\nsubseteq V(C_n)$, for
a cycle $C_n$, factor $v$ contains at least one vertex connected
by an edge with $C_n$. Thus, the result follows.
\end{proof}

\begin{coll}\label{dim}
    If $\Theta$ is an oriented graph with the cycles $C_{1}, \ldots, C_{k}$ such that the corresponding Hecke--Kiselman algebra has finite Gelfand--Kirillov dimension then $$\GKdim
    A_{\Theta} \leqslant \sum_{j=1}^{k}\left(\sum_{x\in \mathcal{A}_{j}} k_x + 1\right),$$ where
    $\mathcal{A}_{j}$ consists of all vertices of $\Theta$ that are
    connected by an edge with the cycle $C_{j}$ for $j=1,\ldots, k$.

\end{coll}

\begin{proof}
    It is enough to prove that every family of reduced words in
    $A_{\Theta}$ of the form $$w_{1}^*v_1w_{2}^*\ldots
    v_{s-1}w_{s}^*$$
    can be expressed as a regular expression of the form (\ref{dlugosc}) with $s\leqslant \sum_{j=1}^{k}\sum_{x\in \mathcal{A}_{j}} \left(k_x + 1\right)$.

        From Observation~\ref{supp} for every $n$ we have $\supp(w_n)\s V(C_{j(n)})$,
        for some $j(n)\in\{1,\ldots, k\}$ and $w_n$ are factors of the word
        $(x_{N}(x_{1}\ldots x_{i})(x_{N-1}\ldots x_{i+1}))^{\infty}$ of full support,
        where $x_1 \to x_2 \to \ldots \to x_N \to x_1$ is one the cycles $C_{j}$
        with $N = n(j)$ and $i\in\{0,\ldots, N-2\}$.

      By Lemma~\ref{non--cyclic} we can rewrite the considered family of words in such
      a way that between any two $w_{i}$, $w_{j}$ ($i,j\in\{1,\ldots, s\}$)
      such that $\supp(w_i), \supp(w_j)\s V(C_n)$ for some $n\in\{1,\ldots, k\}$
      there is a non-cycle vertex $z$ which is connected by an edge with $C_n$,
      that is $z\in \mathcal{A}_{n}$.

  By Lemma~\ref{oszacowanie}, all vertices $z$ with this property occur at most
  $\sum_{x\in \mathcal{A}_{n}} k_x$ times in total in any reduced word of $A_{\Theta}$.
  Consequently, in the regular expression of the above form, for every $j=1,\ldots, k$,
  factors of the form $w^{*}$ with $\supp(w)\s V(C_{j})$ occur at most
  $\sum_{x\in \mathcal{A}_{j}} k_x +1$ times.

  Because, as already explained, any family of reduced words in
  $A_{\Theta}$ of the form $w_{1}^*v_1w_{2}^*\ldots
  v_{s-1}w_{s}^*$ can be rewritten in such a way that for every $w_i$ we have
  $\supp(w_i)\s V(C_{j})$ for some $j\in\{1,\ldots, k\}$, it follows that
  $s \leqslant \sum_{j=1}^{k}\left(\sum_{x\in \mathcal{A}_{j}} k_x + 1\right)$.

  From Theorem~\ref{automata-est} we know that the set of normal (reduced) words of
  $A_{\Theta}$ is a finite union of regular expressions of the form
  $v_0w_{i_1}^*v_1w_{i_2}^*v_2\ldots v_{s-1}w_{i_s}^*v_{s}$.
  Therefore, from the above reasoning and Theorem~\ref{automata-est} it follows that
  $\GKdim A_{\Theta}\leqslant\sum_{j=1}^{k}\left(\sum_{x\in \mathcal{A}_{j}} k_x + 1\right)$, as claimed.
\end{proof}

Our next step is to construct a family of reduced words of the algebra $A_{\Theta}$ described by a regular expression with exactly $s = \sum_{j=1}^{k}\left(\sum_{x\in \mathcal{A}_{j}} k_x + 1\right)$ stars and such that for different substitutions of stars with positive integers we get different elements. As for every word $w$ we have $w^*w= w^{+}$, we will write $w^{+}$ instead of $w^{*}w$ and we refer to the number of stars in the regular expression even if $+$ is used.

Let us recall that we assume that the set of vertices of $\Theta$ is ordered as in Definition~\ref{order}.

Let $\Theta$ be a graph with cycles $C_{1},\ldots, C_{k}$ of the length
$i_j\geqslant 3$ for $j\in\{1,\ldots, k\}$. Denote by $\Theta'$ the maximal cycle--reachable subgraph of $\Theta$.
We will construct a family of reduced words in $\HK_{\Theta}$ via an insertion process that is described below.

{\bf Step 1.} First we insert subsequent vertices contained in the
cycle--reachable subgraph $\Theta'$ of the graph $\Theta$ that are
not cycle vertices to certain words, starting from the trivial
word $1$. At every step a chosen generator $y$ is inserted at the
beginning of the word and directly after every vertex of the
(previously constructed) word that is connected by an edge with
$y$. Every vertex $y$ occurs exactly $k_y$ times in the
constructed word. Note that at this stage the resulting word is
not necessarily reduced. The procedure is described precisely
as follows.

As $\Theta$ does not contain two different cycles connected by an
oriented path, either there is at least one terminal vertex $y$
with $k_y=1$ or the graph is a disjoint union of cycles
$C_{1},\ldots, C_{k}$. If the latter case holds we set $w'=1$,
where $1$ is a trivial word and go to Step 2.

Now we consider the case when there are some terminal vertices in
$\Theta'$. Note that a vertex $y$ from $\Theta'$ is terminal
exactly if $k_y=1$. Let $y^{(1)}_1 < \ldots < y^{(1)}_{n_1}$ be
the set of all vertices in $\Theta '$ such that $k_{y^{(1)}_i}=1$
and define $$w_1 = y^{(1)}_1 y^{(1)}_2\cdots y^{(1)}_{n_1}.$$

Next, take the biggest (with respect to the order defined in
Definition~\ref{order}) vertex $y^{(2)}\in V(\Theta ')$ that is
not contained in any cycle of the graph and that has not been used
yet in $w_1$. We can assume that all paths between the cycles and
$y^{(2)}$ lead from the cycles into $y^{(2)}$. Otherwise, all such
paths lead from $y^{(2)}$ into the cycles and the reasoning is
symmetric. If for some non-cycle vertex $z\in V(\Theta ')$ we have
$y^{(2)}\rightarrow z$, then $k_{z} < k_{y^{(2)}}$ and thus
$y^{(2)} < z$. By the choice of $y^{(2)}$ it follows that $z\in\{
y^{(1)}_1, \ldots, y^{(1)}_{n_1}\}$. Moreover, there are exactly
$k_{y^{(2)}}-1$ (recall that $k_{y^{(2)}}$ is the number of paths
starting at $z$) generators in $w_1$ that are connected by an edge
with $y^{(2)}$. Let $w_2$ be the word that is formed from $w_1$ by
    inserting the generator $y^{(2)}$ in such a way that it is the first letter of $w_2$
    and $y^{(2)}$ also directly
    follows in $w_2$ every $y^{(1)}_j$ that is connected by an edge
    $y^{(2)}\rightarrow y^{(1)}_j$ with $y^{(2)}$ in $\Theta '$.
    Generator $y^{(2)}$ occurs in $w_2$ exactly $k_{y^{(2)}}$ times.
    Additionally, every generator $z$ used in the word $w_2$ occurs in this word
    exactly $k_{z}$ times.

Similarly, if we have already constructed the word $w_i$ for some
$i>1$, then in the next step we insert to this word several copies
of the largest non-cycle generator $y^{(i+1)}\in V(\Theta ')$
that is not in the support of $w_i$ yet. In the word $w_i$ every
generator $z$ occurs $k_z$ times. We know that every $z$ such that
$y^{(i+1)} < z$ is already in the support of $w_i$. In particular
every generator $z$ for which $k_z < k_{y^{(i+1)}}$ is in $w_i$.
As explained above, we can assume that all directed paths
connecting the cycles and $y^{(i+1)}$ start from the cycles.
Therefore, if we have $y^{(i+1)}\rightarrow p$ in the graph
$\Theta'$, then $p\in \supp(w_{i})$. Define the word $w_{i+1}$ by
inserting $y^{(i+1)}$ to $w_{i}$ at the beginning and also
directly after every generator $z\in\supp(w_i)$ such that
$y^{(i+1)}\rightarrow z$ in $\Theta '$. In such a word
    $w_{i+1}$ the element $y^{(i+1)}$ occurs exactly
    $\sum\limits_{y^{(i+1)}\rightarrow z} k_z + 1$ times.
    Let us note that all paths starting at $y^{(i+1)}$ in the graph
    $\Theta$ are either the path of length $0$ or are uniquely determined
    by a path starting at $z$ for some $z$ such that $y^{(i+1)}\rightarrow z$.
    Consequently, in the word
$w_{i+1}$ the element $y^{(i+1)}$ occurs exactly
$\sum\limits_{y^{(i+1)}\rightarrow z} k_z + 1 = k_{y^{(i+1)}}$
times.

 After finitely many steps as described above we get a word
$w'$ whose support contains every non-cycle generator $z$ of
$\Theta'$ and with the property that every $z\in\supp(w')$ occurs
in $w'$ exactly $k_{z}$ times.

{\bf Step 2.} Now we insert cycle vertices into the word $w'$
constructed in Step 1. The idea relies on a slight modification of
the previous Step. Namely, we insert regular expressions of the
form $w_{0}w^*w_{1}$ with $\supp(w_0),\supp(w_1),\supp(w)\s
V(C_j)$ ($w_0$ and $w_1$ vary depending on the insertion place),
for a cycle $C_j$, at the beginning of the constructed regular
expression and directly after every vertex connected by an edge
with $C_j$. The procedure is repeated for every cycle, starting from
the cycle with the biggest vertices in the sense of ordering from
Definition~\ref{order}. It can be precisely described as follows.

For every cycle $C_{i}$ ($i=1,\ldots, k$) with vertices $x_{1,i},\ldots, x_{n, i}$
for some $n\geqslant 3$ denote by $c_{i}$ the reduced word of the form
$x_{1,i}\cdots x_{n, i}$.

We can write $w'=v_1\cdots v_{m+1}$, where every $v_i$ is the
word of minimal possible length that ends with an element $z_{i}$
connected by an edge with the cycle $C_k$ (possibly with
$v_{m+1}=1$) for $i=1,\ldots, m$. Note that we have $m=\sum_{x\in
\mathcal{A}_{k}} k_x$ if $\mathcal{A}_{k}$ is non-empty and $m=0$
otherwise.

For every vertex $z_i$ connected by an edge with the cycle $C_k$ of length $n$, we may choose $j(i)\in\{1,\ldots, n\}$ such that either $z_{i}\rightarrow x_{j(i), k}$ or $x_{j(i), k}\rightarrow z_{i}$. Then we define the regular expression (that is certain family of words) $r_{k}$ as follows:
\begin{align*}
    c_{k}^{+} (x_{1,k}\ldots x_{j(1)-1,k}) v_1 (x_{j(1),k}\cdots x_{n, k})
    c_{k}^{+}(x_{1,k}\ldots x_{j(2)-1, k}) \cdots \\ \cdots c_{k}^{+}(x_{1, k}\ldots x_{j(m)-1, k})
    v_{m} (x_{j(m), k}\cdots x_{n, k}) c_{k}^{+}v_{m+1}.
\end{align*}
In this expression Kleene star $*$ occurs exactly $m_k =
\sum_{x\in \mathcal{A}_k} k_x + 1 $ times, where $\mathcal{A}_k$
consists of all vertices $x$ that are connected by an edge with
the cycle $C_k$ in $\Theta'$. If $\mathcal{A}_k$ is empty, that is
there are no vertices connected by an edge with the cycle $C_k$
and $w'=v_1$ we define the regular expression $r_1$ as
$c_{k}^{+}v_1$. Then we also assume that $\sum_{x\in
\mathcal{A}_k} k_x = 0$ and thus Kleene star $*$ occurs exactly $1
= \sum_{x\in \mathcal{A}_k} k_x + 1$ times.

Next we repeat this procedure for every cycle of the graph
$\Theta$. More precisely, at every step we rewrite the constructed
regular expression $r_j$ as $v_1\cdots v_{m+1}$, where
$v_1,\ldots, v_{m}$ are regular expressions of minimal possible
length that end with an element $z_{i}$ connected by an edge with
the cycle $C_{j-1}$ (perhaps with $v_{m+1}=1$). If there are no
vertices connected by an edge with $C_{j-1}$, we set $r_j=v_{1}$,
that is $m=0$. Note that we have $m=\sum_{x\in \mathcal{A}_{j-1}}
k_x$, where for empty $\mathcal{A}_{j-1}$ we put $\sum_{x\in
\mathcal{A}_{j-1}} k_x = 0$. For every vertex $z_i$ connected by
an edge with the cycle $C_{j-1}$ of length $n$, we may choose
$j(i)\in\{1,\ldots, n\}$ such that either $z_{i}\rightarrow
x_{j(i), j-1}$ or $x_{j(i), j-1}\rightarrow z_{i}$. Then define
the regular expression $r_{j-1}$ as:
\begin{align}\label{family}
    c_{j-1}^+ (x_{1,j-1}\ldots x_{j(1)-1,j-1}) v_1 (x_{j(1),j-1}\cdots x_{n, j-1})
    c_{j-1}^+(x_{1,j-1}\ldots x_{j(2)-1, j-1}) \cdots \\
    \cdots c_{j-1}^+(x_{1, j-1}\ldots x_{j(m)-1, j-1})
    v_m (x_{j(m), j-1}\cdots x_{n, j}) c_{j-1}^+ v_{m+1}.\nonumber
\end{align}
As before, if $\mathcal{A}_{j-1}$ is empty, we set $r_{j-1} =
c_{j-1}^{+}r_j$. Then expression $r_{j-1}$ contains exactly
$m_{j-1} = m_j + \sum_{x\in \mathcal{A}_{j-1}} k_x + 1$ Kleene
stars.

This way we construct a regular expression $r_1$ that contains exactly $m_1= m_2 +\sum_{x\in\mathcal{A}_{1}} k_x + 1 = \sum_{j=1}^{k}\left(\sum_{x\in \mathcal{A}_{j}} k_x + 1\right) $ stars. We will show that $r_1$, treated as a family of words, consists of reduced words of $\HK_{\Theta}$. This will be crucial to get the lower bound for the Gelfand--Kirillov dimension of the algebra~$A_{\Theta}$.

\begin{lemma}\label{lower-bound}
    Words \eqref{family} are reduced in $A_{\Theta}$ with respect to
    the system introduced in Theorem~\ref{basis}. Consequently,
    $\GKdim A_{\Theta}\geqslant\sum_{j=1}^{k}\left(\sum_{x\in \mathcal{A}_{j}} k_x + 1\right)$.
\end{lemma}

\begin{proof}
    We will show that no leading term of reductions of the form
    (i)--(iii) listed in Theorem~\ref{basis} appears as a factor of a
    word $w$ from the family described by the regular expression $r_1$.

    Reductions of type (i) and (ii).
    First consider any factor of $w$ of the form $tvt$ for some
    generator $t$ and any word~$v$ such that $t\notin\supp(v)$.
    We need to show that then there are vertices $x,y\in\supp(v)$ such that
    $x\rightarrow v$ and $v\rightarrow y$.

    Assume first that $t$ is a cycle vertex, let $t\in V(C_j)$ for a cycle $C_j$
    with vertices $x_{1, j}, \ldots, x_{n, j}$ and some $j\in\{1,\ldots, k\}$.
    Consider the image of elements of the family described by a regular expression
    \eqref{family} under the natural projection
    $\varphi_j:\HK_{\Theta}\rightarrow \HK_{C_j}$ onto the
    Hecke--Kiselman monoid associated to the cycle $C_j$,
    such that $\varphi_j(x) = 1$ for all $x\notin V(C_j)$.

    By the construction, every such image is a factor of
    $(x_{1, j} \cdots x_{n, j})^{\infty}$. Thus if $t$ is a cycle
    vertex $x_i$, then $x_{i-1}, x_{i+1}\in\supp(v)$, where for $i=1$
    and $i=n$ we set $i-1=n$ and $i+1=1$, respectively. In particular
    it is then impossible to have $t\nrightarrow v$ or $t\nleftarrow
    v$.
    Therefore, we may consider any $t$ that is not in the cycle
    and we claim that in every factor $tvt$ the set $\supp(v)$
    contains elements $p$ and $q$ connected by an edge with $t$ such
    that $t\rightarrow p$ and $q\rightarrow t$.

    Note that every sink or source vertex $x$ either is not contained in the maximal cycle--reachable subgraph $\Theta'$ of the graph or $k_x=1$. Consequently, it occurs at most once in every word described by the considered regular expression. Thus we know that $t$ is neither a sink nor a source vertex.

    Now assume that $t$ is non-cycle and not terminal vertex,
    see Section~\ref{prelim}, from
    $\Theta'$. Assume first that all oriented
    paths connecting $t$ with the cycles lead from the cycles to~$t$.
    For any $z\rightarrow t$ contained in the graph $\Theta'$ we have $z < t$.
    From the construction of
    the family of words it follows that such $z$ is inserted into the word
    between any two occurrences of $t$, that is $z\in\supp(v)$ and
    the leading term from the reduction (i) in Theorem~\ref{basis} is impossible.
    The other way round, the generator $t$
    is inserted into the regular expression at the beginning and directly after any vertex $y$
    such that $t\rightarrow y$ ($y$ are inserted before $t$). In particular, all such generators $y$ occur between any two $t$'s. It follows directly that no leading
    term of a reduction of type (ii) appears as a factor of $w$. The case when all oriented paths lead from $t$ to the cycles can be handled in much the same way.

    Reductions of type (iii).
    We claim that $w$ does not contain any factor
    $t_1vt_2$ such that $t_1 > t_2$ and $t_2 \nleftrightarrow t_1v$.
    If $t_1$ is contained in any of the cycles, then $t_1 > t_2$ implies that
    also $t_2$ is a cycle vertex.

 Let a word $w$ be described by a regular expression (\refeq{family}).
 By the construction, for every factor of $w$ of the form $px_{i,j}$,
 where $x_{i,j}$ is a cycle vertex and $p$ is a word such that
 $p\nleftrightarrow x_{i, j}$, the word $p$ consists of cycle vertices $x_{l, m}$
 such that $m < j$. In particular we have $g < x_{i, j}$ for every $g\in\supp(p)$.
 Thus there is no factor of the above form with $t_2$ being a cycle element.

    In consequence, we can assume that both $t_1$ and $t_2$ are
    non-cycle vertices.

    We claim that no word $w_i$ from the first part of the construction of regular expression $r_1$ has a factor of type (iii) from
    Theorem~\ref{basis}.
    To do so, we proceed by induction on $i$.
    First observe that the assertion holds for $i=1$, as generators in $w_1$ are in the increasing order. Hence, assume that the
    claim holds for some $w_i$ and denote by $y^{(i+1)}$ the vertex inserted in the next step, that is $\supp(w_{i+1})\setminus \supp(w_{i})=\{ y^{(i+1)}\}$. Then
        every factor $t_1vt_2$ such that $t_1
    > t_2$ and $t_2 \nleftrightarrow t_1v$ in $w_{i+1}$ would have
    $t_2 = y^{(i+1)}$ because by the inductive hypothesis
    $w_{i}$ does not have such factors and all elements of
    $\supp(w_{i})$ are bigger than $y^{(i+1)}$. On the other hand, in
    $w_{i+1}$ the element directly before $y^{(i+1)}$ is connected by
    an edge with $y^{(i+1)}$. Thus in $w_{i+1}$ every factor of the
    form $t_1 v y^{(i+1)}$ with $t_1 > y^{(i+1)}$ is such that the last
    generator of $t_1 v$ is connected by an edge with $y^{(i+1)}$. The
    inductive assertion holds.

    Consequently, we know that the word $w'$, built in the first step of the construction, does not contain factors of type (iii). The regular expression $r_1$ is obtained from $w'$ by inserting only cycle generators. Every factor $t_2vt_1$ with $t_2
    > t_1$ and $t_2 \nleftrightarrow t_1w$ would therefore start or
    end with a cycle vertex, that is either $t_1$ or $t_2$ is a cycle
    vertex. This is not possible as we explained earlier. We have
    proved that any $w$ described by the regular expression $r_1$ does not contain
    factors of the form (iii) in the Theorem~\ref{basis}, as claimed.
    The first part of lemma follows.

    As every word described by the regular expression $r_1$ is reduced, two different words are equal in the algebra $A_{\Theta}$ if and only if they are equal as elements of free monoid generated by the vertices of~$\Theta$.

    Let us notice that every element $w$ of this family of words is uniquely determined
    by $m$ positive integers $(n_1,\ldots, n_m)$, where
    $m=\sum_{j=1}^{k}\left(\sum_{x\in \mathcal{A}_{j}} k_x + 1\right)$,
    such that $n_1,\ldots, n_m$ are powers of consecutive cycles of the form
    $(x_{1, j} \cdots x_{n, j})$ that correspond exactly to stars $*$.
    If a family is of the form $v_0w_{i_1}^*v_1w_{i_2}^*v_2\ldots v_{m-1}w_{i_m}^*v_{m}$,
    denote by $q$ the length of the word $v_0v_1\ldots v_{m-1}v_{m}$ and let $K$
    be the maximal length of cyclic subgraph in $\Theta$.
    Then the number of elements of length at most $n$ in this family,
    denoted by $d(n)$ for $n\geqslant 1$, is not smaller than the number of
    elements of the set
    $\{(n_1,\ldots, n_m): n_i\in\mathbb{Z}_{+}, n_1 + \cdots + n_m \leqslant \frac{n-q}{K}\}$.
    It follows that for almost all $n$ we have $d(n)\geqslant d_{m}(Cn)$
    for certain constant $C$, where $d_m$ is the growth function of polynomials in $m$
    variables. Consequently, from Example~1.6 in \cite{krause}, it follows that
    $\GKdim A_{\Theta}\geqslant\sum_{j=1}^{k}\left(\sum_{x\in \mathcal{A}_{j}} k_x + 1\right)$.
\end{proof}

Corollary~\ref{dim} and Lemma~\ref{lower-bound} are
summarized in the following theorem that describes the
Gelfand--Kirillov dimension of the Hecke--Kiselman algebra associated
to any oriented graph without two different cycles connected by an oriented path.

\begin{theorem}\label{main} Let $\Theta$ be a finite simple oriented graph with the cycles $C_1,\ldots, C_k$ for some $k\geqslant 1$ without two different cycles connected by an oriented path. In particular, for any non-cyclic vertex $x$ connected by an oriented path with a cycle either all paths between $x$ and cycles are directed from $x$ into the cycles or all begin at the cycles. Denote by $\mathcal{A}_{j}$ the set of vertices of the graph that are connected by an edge with the cycle $C_j$ for $j=1,\ldots, k$. For any $x\in \mathcal{A}_{j}$ let $k_x$ be the number of oriented paths of length $\geqslant 0$ in $\Theta$ that start with $x$ if all paths between $C_j$ and $x$ start with the cycle vertices and oriented paths that end with $x$ otherwise. Then $$\GKdim
A_{\Theta} = \sum_{j=1}^{k}\left(\sum_{x\in \mathcal{A}_{j}} k_x + 1\right),$$
where $\sum_{x\in \mathcal{A}_{j}} k_x + 1$ is equal to $1$ if $\mathcal{A}_j$ is an empty set.
Lastly, if the graph $\Theta$ does not contain any cycle, then $\GKdim
A_{\Theta} = 0$.

\end{theorem}

\section{An example}
Let us illustrate concepts from Theorem~\ref{main} and its proof for the oriented graph $\Theta$ presented in the picture.
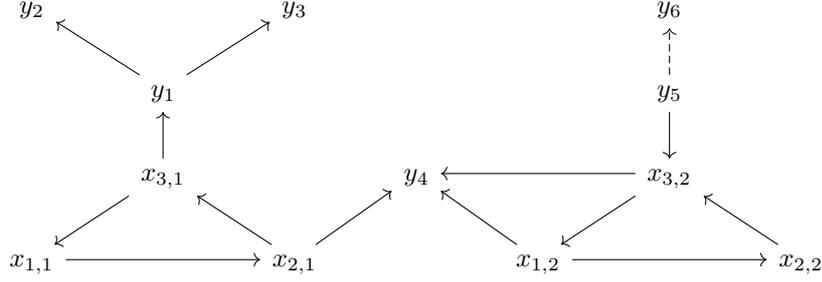
\begin{center}
\begin{tikzcd}
   &                                &                                 & \Theta &                                 &                                 &                      \\
   y_2                  &                                & y_3                             &        &                                 & y_6                             &                      \\
   & y_1 \arrow[lu] \arrow[ru]      &                                 &        &                                 & y_5 \arrow[d] \arrow[u, dashed] &                      \\
   & {x_{3,1}} \arrow[ld] \arrow[u] &                                 & y_4    &                                 & {x_{3,2}} \arrow[ld] \arrow[ll] &                      \\
   {x_{1,1}} \arrow[rr] &                                & {x_{2,1}} \arrow[ru] \arrow[lu] &        & {x_{1,2}} \arrow[lu] \arrow[rr] &                                 & {x_{2,2}} \arrow[lu]
\end{tikzcd}
\end{center}
The maximal cycle--reachable subgraph $\Theta ' $ is the full
subgraph of $\Theta$ with all vertices except $y_6$. The edges of
$\Theta '$ are denoted by solid arrows, whereas the complement is
denoted by dashed ones.

For the non-cycle vertices in $\Theta '$ named as in the picture
we have: $k_{y_{2}}=k_{y_3} = k_{y_{4}} = k_{y_{5}} = 1$ and
$k_{y_1} = 3$. Denote the cycle with vertices $x_{i, 1}$, $i=1, 2,
3$ by $C_1$ and let $C_2$ be the cycle $x_{1, 2}\rightarrow x_{2,
2}\rightarrow x_{3, 2}\rightarrow x_{1,2}$. Then the sets
$\mathcal{A}_{1}$ and $\mathcal{A}_2$ consisting of the vertices
connected by an edge with the cycles are $\mathcal{A}_1 = \{y_{1},
y_{4}\}$ and $\mathcal{A}_2 = \{y_{4}, y_{5}\}$. We get that
$\sum_{x\in \mathcal{A}_{1}} k_x + 1 = 5$ and $\sum_{x\in
\mathcal{A}_{2}} k_x + 1 = 3$.

From Theorem~\ref{main} we obtain the following corollary.
\begin{coll}The Gelfand--Kirillov dimension of the Hecke--Kiselman algebra $A_{\Theta}$ associated to the graph $\Theta$ as in the picture is $8$.
\end{coll}

Following Lemma~\ref{lower-bound} let us construct a family of
reduced words in $A_{\Theta}$ described by a regular expression
with exactly $8$ Kleene stars.

In the set of vertices of $\Theta$ we introduce the following order.
\begin{itemize}
    \item Cycle vertices are such that $x_{1,1} < x_{2, 1} < x_{3, 1} < x_{1, 2} < x_{2, 2} < x_{3, 2}$.
    \item For non-cyclic vertices we may choose any order such that $y_1$ is the smallest one. Assume that $y_1 < y_{2} <y_{3} <y_{4} < y_{5} <y_{6}$.
    \item All cycle vertices are smaller than non-cyclic ones, that is $x_{3, 2} < y_{1}$.
\end{itemize}

Then the word $w'$ without cycle vertices built in the first part
of the construction is of the form $ y_{1}y_{2} y_{1} y_{3}
y_{1}y_{4}y_{5}$. Note that each element $y_j$ of the support of
this word occurs in it exactly $m_{y_j}$ times. Next denote by
$c_i$ the word $x_{1,i}x_{2,i}x_{3,i}$ for $i=1, 2$. We have that
every vertex of $c_1$ is smaller than any vertex of $c_2$. The
regular expression $r_2$ is $ c_{2}^{+} y_{1}y_{2} y_{1} y_{3}
y_{1}y_{4}c_{2}^{+}x_{1,2}x_{2, 2}y_{5}x_{3,2} c_{2}^{+}$.
Finally, the regular expression $r_1$ with exactly $8$ stars and
consisting of reduced words has the following form:
$$(c_{1}^{+}x_{1,1}x_{2,1}) (c_{2}^{+}) \underline{y_{1}}(x_{3,1} c_{1}^{+}x_{1,1}x_{2,1})\underline{y_{2} y_{1}}(x_{3,1} c_{1}^{+} x_{1,1}x_{2,1}) \underline{y_{3} y_{1}}(x_{3,1} c_{1}^{+}x_{1,1})\underline{y_{4}}(x_{2,1}x_{3,1}c_{1}^{+})(c_{2}^{+} x_{1, 2} x_{2, 2})\underline{y_{5}} (x_{3,2} c_{2}^{+}).$$
The consecutive factors of $w'$ constructed in the first step are
underlined for clarity.\\

\noindent {\bf Acknowledgments.} I am very grateful to Jan Okni\'nski for careful reading and many suggestions on the earlier versions of the paper, as well as his continuous support. I would also like to thank Arkadiusz M\k{e}cel for fruitful discussions about the research topic.

\noindent This work is supported by grant 2021/41/N/ST1/03082 of the National Science Centre (Poland).

 \vspace{20pt}

\noindent Magdalena Wiertel \\
email address: \texttt{m.wiertel@mimuw.edu.pl}\\
Institute of Mathematics \\
 University of Warsaw \\
Banacha 2 \\
02-097 Warsaw, Poland \\

\end{document}